\documentclass[12pt,leqno]{amsart}
\usepackage{amssymb,amsmath,amsthm,bm,upgreek}
\usepackage{graphicx}
\usepackage{color}
\usepackage[textsize=tiny]{todonotes}
\usepackage[normalem]{ulem}
\usepackage[bookmarksopen,bookmarksdepth=3,colorlinks,citecolor=red,pagebackref,hypertexnames=true]{hyperref}
\usepackage[msc-links,nobysame,non-sorted-cites, initials]{amsrefs}
\usepackage[inline]{enumitem}
\usepackage{mathtools}
\mathtoolsset{showonlyrefs}
\usepackage{float}
\usepackage{pgfplots}
\usepackage{tikz-3dplot}

\definecolor{darkblue}{RGB}{0,0,160}

\usepackage[lining]{libertine}   
\usepackage{cabin}
\usepackage[libertine]{newtxmath}

\usepackage[T1]{fontenc}

\def\eps{\varepsilon}

\def\d{{\rm d}}

\def\R {\mathbb{R}}

\def\M {{\mathrm M}}

\def\Z {{\mathbb Z}}
\def\1 {{\mbox{\boldmath 1}}}

\def \and{\quad\text{and}\quad}

\newcommand{\cic}{\bm}

 \def\Xint#1{\mathchoice
	   {\XXint\displaystyle\textstyle{#1}}%
	   {\XXint\textstyle\scriptstyle{#1}}%
	   {\XXint\scriptstyle\scriptscriptstyle{#1}}%
	   {\XXint\scriptscriptstyle\scriptscriptstyle{#1}}%
	   \!\int}
	 \def\XXint#1#2#3{{\setbox0=\hbox{$#1{#2#3}{\int}$}
	     \vcenter{\hbox{$#2#3$}}\kern-.5\wd0}}
	 
	 \def\avgint{\Xint-}

\def \no#1#2#3 {{\bf #1} (#3), #2.}
\def \eds#1#2#3 {#1, #2, #3.}

\newcounter{counter}

\numberwithin{equation}{section}
\numberwithin{counter2}{section}
\newtheorem{proposition}[subsection]{Proposition}
\newtheorem{theorem}[counter]{Theorem}

\newtheorem{lemma}[subsection]{Lemma}

\theoremstyle{definition}
\newtheorem{definition}[subsection]{Definition}
\newtheorem*{remark*}{Remark}
\newtheorem*{warn*}{A word of warning}

\newtheorem{remark}[subsection]{Remark} 
\theoremstyle{plain}

\numberwithin{figure}{section}

\makeatletter
\let\save@mathaccent\mathaccent
\newcommand*\if@single[3]{%
  \setbox0\hbox{${\mathaccent"0362{#1}}^H$}%
  \setbox2\hbox{${\mathaccent"0362{\kern0pt#1}}^H$}%
  \ifdim\ht0=\ht2 #3\else #2\fi
  }
\newcommand*\rel@kern[1]{\kern#1\dimexpr\macc@kerna}
\newcommand*\widebar[1]{\@ifnextchar^{{\wide@bar{#1}{0}}}{\wide@bar{#1}{1}}}
\newcommand*\wide@bar[2]{\if@single{#1}{\wide@bar@{#1}{#2}{1}}{\wide@bar@{#1}{#2}{2}}}
\newcommand*\wide@bar@[3]{%
  \begingroup
  \def\mathaccent##1##2{%
    \let\mathaccent\save@mathaccent
    \if#32 \let\macc@nucleus\first@char \fi
    \setbox\z@\hbox{$\macc@style{\macc@nucleus}_{}$}%
    \setbox\tw@\hbox{$\macc@style{\macc@nucleus}{}_{}$}%
    \dimen@\wd\tw@
    \advance\dimen@-\wd\z@
    \divide\dimen@ 3
    \@tempdima\wd\tw@
    \advance\@tempdima-\scriptspace
    \divide\@tempdima 10
    \advance\dimen@-\@tempdima
    \ifdim\dimen@>\z@ \dimen@0pt\fi
    \rel@kern{0.6}\kern-\dimen@
    \if#31
      \overline{\rel@kern{-0.6}\kern\dimen@\macc@nucleus\rel@kern{0.4}\kern\dimen@}%
      \advance\dimen@0.4\dimexpr\macc@kerna
      \let\final@kern#2%
      \ifdim\dimen@<\z@ \let\final@kern1\fi
      \if\final@kern1 \kern-\dimen@\fi
    \else
      \overline{\rel@kern{-0.6}\kern\dimen@#1}%
    \fi
  }%
  \macc@depth\@ne
  \let\math@bgroup\@empty \let\math@egroup\macc@set@skewchar
  \mathsurround\z@ \frozen@everymath{\mathgroup\macc@group\relax}%
  \macc@set@skewchar\relax
  \let\mathaccentV\macc@nested@a
  \if#31
    \macc@nested@a\relax111{#1}%
  \else
    \def\gobble@till@marker##1\endmarker{}%
    \futurelet\first@char\gobble@till@marker#1\endmarker
    \ifcat\noexpand\first@char A\else
      \def\first@char{}%
    \fi
    \macc@nested@a\relax111{\first@char}%
  \fi
  \endgroup
}
\makeatother

\begin{document}

\title[Singular integrals along lacunary directions in $\mathbb R^n$]{singular integrals along lacunary directions in $\mathbb R^n$}

\author[N. Accomazzo]{Natalia Accomazzo}
\address{Departamento de Matem\'aticas, Universidad del Pais Vasco, Aptdo. 644, 48080
Bilbao, Spain\newline \indent
Department of Mathematics, The University of British Columbia, Room 121, 1984 Mathematics Road,
\newline\indent Vancouver, BC, Canada V6T 1Z2}
\email{\href{mailto:naccomazzo@math.ubc.ca}{\textnormal{naccomazzo@math.ubc.ca}}}
\thanks{N. Accomazzo was partially supported by projects MTM2017-82160-C2-2-P and PGC2018-094528-B-I00 (AEI/FEDER, UE)}

\author[F. Di Plinio]{Francesco Di Plinio} \address{Department of Mathematics, Washington University in St. Louis, One Brookings Drive, St. Louis, \newline \indent MO 63130-4899, USA}
\email{\href{mailto:francesco.diplinio@wustl.edu}{\textnormal{francesco.diplinio@wustl.edu}}}

\thanks{F. Di Plinio was partially supported by the National Science Foundation under the grants   NSF-DMS-1650810 and NSF-DMS-1800628/2000510}

\author[I. Parissis]{Ioannis Parissis}
\address{Departamento de Matem\'aticas, Universidad del Pais Vasco, Aptdo. 644, 48080 Bilbao, Spain and \newline \indent Ikerbasque, Basque Foundation for Science, Bilbao, Spain}

\email{\href{mailto:ioannis.parissis@ehu.es}{\textnormal{ioannis.parissis@ehu.es}}}
\thanks{I. Parissis is partially supported by project PGC2018-094528-B-I00 (AEI/FEDER, UE) with acronym “IHAIP”, grant T1247-19 of the Basque Government and IKERBASQUE}


\begin{abstract} A recent result by Parcet and Rogers is that finite order lacunarity characterizes the boundedness of the  maximal averaging operator associated to an infinite set of directions in $\mathbb R^n$. Their proof is based on geometric-combinatorial coverings of fat hyperplanes by two-dimensional wedges. Seminal results by Nagel-Stein-Wainger relied on geometric coverings of $n$-dimensional nature. In this article we  find the sharp cardinality estimate for singular integrals along finite subsets of finite order lacunary sets in all dimensions. Previous results only covered the special case of the directional Hilbert transform in dimensions two and three. The proof is new in all dimensions and relies, among other ideas, on a precise covering of the $n$-dimensional Nagel-Stein-Wainger cone by two-dimensional Parcet-Rogers wedges. 
\end{abstract}
\maketitle

\section{Introduction}

We study sharp cardinality bounds for maximal singular integrals along lines in general ambient Euclidean dimension, when the allowed set of lines is constrained to not support Besicovitch sets. Our main focus is thus on directional singular integrals, defined via the Fourier transform 
as follows.
Let $m$ be a H\"ormander-Mikhlin multiplier on $\mathbb{R}$, that is,
\[
m \in C^{\infty}(\R\backslash \{0\}), \qquad  \sup_{\xi \in \R \setminus \{0\}} |\xi|^\alpha|\partial^{\alpha}m(\xi)|\lesssim_{\alpha} 1, \qquad \forall\alpha\geq 0.\]
For $f\in C^{\infty}_0(\mathbb{R}^n)$ and $v \in {\mathbb{S}^{n-1}}$ consider the directional multiplier
\begin{equation}
\label{e:defTv}
T_vf(x)\coloneqq \int_{\mathbb{R}^n}\widehat{f}(\xi)m(\xi \cdot v)e^{ix\cdot \xi} \, \d\xi,\qquad x\in\R^n.
\end{equation}
Of course, $T_v$ depends on the choice of symbol $m$.  We henceforth suppress this dependence from the notation as the multiplier $m$ may be thought of as fixed throughout the exposition. A most relevant choice is that of the analytic projection $m=\cic{1}_{(0,\infty)}$. In that case, up to a linear combination with the identity operator, $T_v$ is the Hilbert transform along the direction $v$. 
 
For each fixed $v$, $L^p(\R^n)$-boundedness of the directional multiplier $f\mapsto T_v f$ is an immediate consequence of a fiberwise application of the $L^p(\R)$-bound for the  one-dimensional multiplier operator $f\mapsto (m\widehat f)^\vee$ and Fubini's theorem. On the other hand,  $L^p$-bounds for the operator 
\[
f\mapsto T_{v(x)} f(x), \qquad x\in \R^n,
\]
where the directional multiplier is applied along a \emph{variable} choice of lines $x\mapsto v(x)$, are highly nontrivial. The latter question, posed by E.\ Stein during his 1986 ICM plenary lecture \cite{SteinICM}, was initially motivated by the analogy with the corresponding  $L^p$-boundedness problem for the \emph{maximal  averaging operator} along a vector field $v$, which plays the role of the Hardy-Littlewood maximal operator in the context of $L^p$-differentiation along variable lines. The critical Lebesgue exponent is $p=n$, dictated by the existence of   Besicovitch sets of measure zero. Testing on one such set yields the necessary condition that the choice of lines $v$ be a Lipschitz function, and that either the multiplier $m$ or the averaging operator be suitably truncated to spatial scales smaller than the inverse of $\|v\|_{\mathrm{LIP}}$. Whether this condition is also sufficient, at least for weak $L^2$-boundedness in dimension two, is the object of an earlier conjecture  of Zygmund. 

Partial results towards Zygmund's conjecture are due to Bourgain \cite{BOUR}; see also Guo \cite{Guo2}. Partial progress on $L^p$-bounds for the truncated Hilbert transform along a Lipschitz vector field has been obtained, among others, by Lacey and Li \cite{LacLi:mem,LacLi:tams}, Stein and Street \cite{StStMRL},  Bateman and Thiele \cite{BT}, Guo, Thiele, Zorin-Kranich with the second author \cite{DPGTZK}. {We also note that   Demeter   \cite{Dem} proved the sharp $L^2$-bounds for maximal directional H\"ormander-Mikhlin multipliers along finite but arbitrary sets of directions. The proof in \cite{Dem} relies strongly on the vector field result of \cite{LacLi:tams} and the Chang-Wilson-Wolff reduction, the latter of which    we also use in the present paper.}

An alternative way of ensuring $L^p$-bounds for maximal directional averages, and the ensuing differentiation theorems, is to require that the range $\Omega$ of the vector field $v(x)$ does not support Besicovitch sets. In two dimensions, the infinite sets $\Omega$ giving rise to an $L^p$-bounded maximal directional averaging operator  have been fully characterized as finite unions of finite order lacunary sets. The sufficiency in the full range is due to Sj\"ogren and Sj\"olin \cite{SS}, building upon techniques of Nagel, Stein and Wainger \cite{NSW}. The harder necessity statement is due to Bateman \cite{BAT}. In higher dimensions an analogous characterization was only recently achieved by Parcet and Rogers \cite{PR2}. Lacunary sets of directions in the plane appear for instance in the seminal article by C\'ordoba and R.\ Fefferman \cite{CorFef}, as well as in the already mentioned   \cite{NSW,SS}, among many others. The correct generalization to higher dimension is, loosely speaking, as follows: a set $\Omega $ is lacunary if the projection of $\Omega$  on each two-dimensional subspace spanned by a pair  of coordinate vectors is a two-dimensional lacunary set. This definition, detailed in Section \ref{s2} and appearing for the first time in \cite{PR2}, encompasses the previously known examples of \cite{NSW} and of Carbery \cite{Carbery}.   

As anticipated, the main result of this article  is the full singular integral analog of the Parcet-Rogers result.   In particular, we completely close the question, raised for instance in \cite[Section 4]{PR2}, of sharp $L^p(\R^n)$-bounds for the maximal directional multiplier operator
\[
	T_{O}f (x)\coloneqq \sup_{v\in O}|T_vf(x)|,\qquad x\in\R^n,
\]
when $O $ is a finite subset of a finite order lacunary set $\Omega$. Here, sharpness is referred to the dependence of the operator norm of $T_O$ on the cardinality of $O$. In fact, $T_{O}$ is unbounded on every $L^p(\R^n)$ when $O$ is infinite and a   lower bound   $\|T_O\|_{L^p}\gtrsim \sqrt{\log \# O}$ holds for every finite set when $m=\cic{1}_{[0,\infty)}$: this is a result of Laba, Marinelli and Pramanik \cite{LMP}, elaborating on the two-dimensional counterexample of  \cite{Karag}. With the precise definition of a lacunary set of direction given in Definition~\ref{def:lac},  the rigorous statement of our main result is the following.
\begin{theorem}\label{thm:main} Let $n\ge 2$, $1<p<\infty$, and $\Omega \subset {\mathbb{S}^{n-1}}$ be a lacunary set of finite order. Then
\begin{equation} \label{e:main}
\sup_{\substack{O \subset \Omega\\\#O=N}}\|T_O f\|_{L^p(\R^n)}\lesssim  (\log N)^{1/2}\|f\|_{L^p(\R^n)},
\end{equation}
where the implicit constants depends on the dimension $n$, on $p$, and on the order of lacunarity of the set $\Omega$. 
\end{theorem}
A comparison with the above mentioned lower  bound \cite{LMP} shows that the $N$-dependence in Theorem \ref{thm:main} is in general best possible.

\begin{remark} Our methods work equally well for the more general case of families of translation invariant directional singular integrals of the form $(R_vf)^\wedge (\xi) \coloneqq m_v(\xi \cdot v)  \widehat f(\xi)$. Here 
\[
{\mathbb{S}^{n-1}}\supset\Omega \ni v \mapsto m_v(\cdot)
\]
is a measurable collection of H\"ormander-Mikhlin multipliers on $\R$ obeying uniform bounds
\[
m_v \in C^\infty(\R\backslash\{0\}),\qquad \sup_{v\in \Omega} \sup_{\xi  \in\R\backslash\{0\}}  |\xi|^{\alpha}|\partial ^\alpha m_v(\xi)|\lesssim_\alpha1 ,\qquad \forall \alpha\geq 0.
\]
 Indeed,   the conclusion of Theorem~\ref{thm:main} holds \emph{verbatim} for the maximal operator
 \[
 R_\Omega f(x)\coloneqq \sup_{v\in\Omega}|R_vf(x)|\] with identical proof. This variation  may be of interest  when dealing with \emph{tree operators} from time-frequency models of directional singular integrals, see for instance \cite{DDP,LacLi:mem}. The corresponding  multipliers differ for each tree, but  they do obey uniform bounds. In the two-dimensional case, maximal  directional multipliers such as $R_\Omega$  have been studied in \cite{KaraLac} for arbitrary finite sets of directions $\Omega\subset S^1$.
\end{remark}

Estimate \eqref{e:main} was proved, in the case of the Hilbert transform only, in dimensions $n=2$ \cite{DPP} and  $n=3$ \cite{DPPimrn}. For $n\geq 4$, the theorem above is new even for the maximal directional Hilbert transform: in fact, Theorem \ref{thm:main} is the first sharp estimate for maximal directional singular integrals in general dimensions. The presence of a generic H\"ormander-Mikhlin symbol $\xi\mapsto m(\xi \cdot v)$ which is not constant in the halfspaces perpendicular to $v$, as well as the availability of more coordinates in dimensions $n\geq 4$, introduce new, and intertwined, essential obstacles that may not be treated with the approach of \cite{DPPimrn,DPP}. 

In fact, the analysis in  \cite{DPPimrn} relied on a model operator for the maximal directional Hilbert transform which may be described heuristically as the maximal truncation to products of two-dimensional  inner-outer wedges from \cite{PR2}. This approach is satisfactory in dimension two and three. However, an adaptation of the counterexamples from \cite{Karag,LMP} yields a lower bound of $(\log N)^{\lfloor \frac n2\rfloor}$ on the $L^p$-norms of the model operator. This is done by constructing   a two-dimensional counterexample  from \cite{Karag} for each of the $\lfloor \frac n2\rfloor$ pairs with distinct entries, out of the $n$ coordinates, in a way that the counterexamples are not interacting with each other; see \cite[Section 6]{DPPimrn} for details. Ultimately, these considerations show that the sharp exponent obtained here in is out of reach for the purely two dimensional approach of \cite{DPPimrn,PR2} and novel ideas are needed. 

The correct approach  in dimension $n\geq 4$ is a new type of geometric covering that combines the  two-dimensional wedges of Parcet-Rogers \cite{PR2} with the full-dimensional cones of Nagel, Stein and Wainger from \cite{NSW}. A rough description of the proof is as follows: we cover the singularity hyperplane $\xi\cdot v=0$  with the exterior of a full dimensional cone. When $v$ comes from a lacunary set, these exterior cones give rise to a bounded square function: this is shown by covering each exterior cone by unions of two-dimensional wedges. The complementary part of the operator is then a maximal conical multiplier which is amenable to a one parameter Littlewood-Paley  square function estimate, via the Chang-Wilson-Wolff inequality. In contrast, the maximal truncation to products of two-dimensional  inner-outer wedges may only be treated with a Littlewood-Paley square function in $\lfloor \frac n2 \rfloor$ parameters, whence the unavoidable $(\log N)^{\lfloor \frac n2\rfloor}$ loss. 

A key component when dealing with higher order lacunarity is the use of recursive-type vector-valued estimates. We find convenient to treat these by means of $L^p(w)$-bounds for directional weights, so that vector-valued estimates follow for free from extrapolation techniques.  These tools are recalled in Section \ref{s:w}. The proof of Theorem \ref{thm:main} is provided in Section \ref{sec:proof}, while the concluding Section \ref{s:q} contains complementary remarks and open questions.

\subsection*{Acknowledgments} Part of this research was carried out during N.\ Accomazzo's two-month research stay at the University of Virginia Mathematics Department, whose   kind hospitality is gratefully acknowledged. F.\ Di Plinio warmly thanks Jongchon Kim for fruitful discussions on  the subject of directional multiplier operators in the plane. 

The authors would like to thank the anonymous referees for an expert reading and suggestions that helped us improve the presentation of the paper.
\section{Lacunary sets of directions and associated frequency projections} \label{s2}
We begin this section  with a thorough definition of  \emph{lacunary sets of directions} in general dimension. We later give a simplified but equivalent version which will be used throughout the paper. In the remainder of the section, we define  frequency projections, associated to lacunary cones or wedges, which will be used to decompose the maximal multipliers along lacunary sets into   tractable pieces.

\subsection{Lacunary sets of directions} Throughout the paper the ambient space is $\R^n$ and we consider sets of directions $\Omega \subset {\mathbb{S}^{n-1}}$. {Note that by possibly adding $O_n(1)$ directions to $\Omega$ we can always assume that $\mathrm{span}(\Omega)=n$; we will do throughout the rest of the paper.} We then define the sets of ordered pairs of indices
\[
\Sigma=\Sigma(n)\coloneqq \{\sigma=(j,k):\, 1\leq j<k\leq n\};
\]
we will typically drop the dependence on $n$ from the notation.

For $\sigma\in\Sigma$ we now consider lacunary sequences $\{\theta_{\sigma,i}\}_{i\in\mathbb{Z}}$ that satisfy $0<\theta_{\sigma,i+1}\le\lambda_{\sigma}\theta_{\sigma,i}$, with $0<\lambda_{\sigma}<1$. Take $\lambda\coloneqq \max_{\sigma}\lambda_{\sigma}$. From here on we will assume that the lacunarity constant $\lambda\in(0,1)$ has a fixed numerical value and all sequences considered below will be lacunary with respect to that fixed value $\lambda$.

Given an orthonormal basis (ONB) of $\mathrm{span}(\Omega)=\R^n$
\[
\mathcal B\coloneqq (e_1,\dots,e_n),
\]
and a choice of lacunary sequences $\{\theta_{\sigma,\ell}\}$ as above we get for each $\sigma\in\Sigma$ a partition of the sphere into sectors
\[
S_{\sigma,\ell}\coloneqq\left\{v\in {\mathbb{S}^{n-1}}:\theta_{\sigma,\ell+1}<\frac{|v\cdot e_{\sigma(2)}|}{|v\cdot e_{\sigma(1)}|}\le \theta_{\sigma,\ell}\right\},\qquad {\mathbb{S}^{n-1}}=\bigcup_{\ell\in\Z}S_{\sigma,\ell}.
\]
Strictly speaking we need to complete the partition by adding the limit set $S_{\sigma,\infty} \coloneqq {\mathbb{S}^{n-1}}\cap (e_{\sigma(1)^\perp} \cup e_{\sigma(2) ^\perp})$. A convenient way to do so is to define $\Z^*\coloneqq \Z\cup\{\infty\}$. We write any $\Omega\subseteq {\mathbb{S}^{n-1}}$ as a disjoint union as follows:
\[
\Omega = \bigcup_{\ell \in\Z^*} \Omega \cap S_{\sigma,\ell}\coloneqq \bigcup_{\ell \in\Z^*} \Omega_{\sigma,\ell},  \qquad \forall \sigma\in \Sigma.
\]
The collection of $|\Sigma(n)|=n(n-1)/2$ partitions of $\Omega$ will be called a \emph{lacunary dissection} of $\Omega$ with parameters $\mathcal B$ and $\{\theta_{\sigma,\ell}\}$. In particular we have that $\{S_{\sigma,\ell}\}$ as defined above is a lacunary dissection of the sphere ${\mathbb{S}^{n-1}}$. We will refer to the sets $\{\Omega_{\sigma,\ell}\}, \{S_{\sigma,\ell}\}$ as \emph{sectors} of a dissection. 

We will also need a finer partition of subsets of the sphere into \emph{cells} which is generated as follows. Consider a lacunary dissection of $\Omega\subseteq {\mathbb{S}^{n-1}}$, namely an ONB $\mathcal B$ and sequences $\{\theta_{\sigma,\ell}\}$. Given $\cic{\ell}=\{\ell_\sigma:\, \sigma\in \Sigma(n)\}\in\Z^\Sigma$ we define
\[
S_{\cic{\ell}}\coloneqq \bigcap_{\sigma\in\Sigma} S_{\sigma,\ell_\sigma	},\qquad \Omega_{\cic{\ell}}\coloneqq \bigcap_{\sigma\in\Sigma} \Omega_{\sigma,\ell_\sigma	},
\]
so that we get the partitions
\[
{\mathbb{S}^{n-1}}=\bigcup_{\cic{\ell}\in \Z^\Sigma}S_{\cic{\ell}},\qquad \Omega_{\cic{\ell}} = \bigcup_{\cic{\ell}\in \Z^\Sigma}\Omega_{\cic{\ell}}. 
\]
{We note here, as in \cite{PR2}*{p. 1540}, that this partition of $\mathbb S^{n-1}$ is in some sense redundant as many of the cells $S_{\cic{\ell}}$ and corresponding sets $\Omega_{\cic{\ell}}$ will be empty. This is unavoidable if one wants a definition of lacunarity in higher dimensions that yields bounded maximal operators.}

We now recall the definition of lacunary sets of directions introduced in \cite{PR2}*{p. 1537}.

\begin{definition}\label{def:lac}Let $\Omega \subset {\mathbb{S}^{n-1}}$ be a set of directions and assume that $\mathrm{span}(\Omega)=\mathbb R^n$. Then $\Omega$ is called lacunary of order $0$ if it consists of a single direction. If $L$ is a positive integer then $\Omega$ is called lacunary of order $L$ if there exists a dissection $\{\Omega_{\sigma,\ell}\}$ of $\Omega$ such that for each $\sigma\in \Sigma(n)$ and $\ell \in \Z^*$, the sector $\Omega_{\sigma,\ell}=S_{\sigma,\ell}\cap \Omega$ is a lacunary set of order $L-1$.  A set $\Omega$ will be called \emph{lacunary} if it is a finite union of lacunary sets of finite order.
\end{definition}

Observe that a set $\Omega$ is lacunary of order $1$ if there exists a dissection $\{\Omega_{\sigma,\ell}\}$ such that each sector $\Omega_{\sigma,\ell}$ contains at most one direction.

We immediately simplify the definition of lacunarity by assuming -without loss of generality- that all dissections are given with respect to lacunary sequences $\theta_{\sigma,\ell}=2^{-\ell}$ for all $\sigma\in\Sigma$, corresponding to $\lambda=1/2$. Furthermore by a standard approximation argument we can dispose of the final set of the partition $\Omega_{\sigma,\infty}$ and work with $\Z$ instead of $\Z^*$. Also, {by a finite splitting}, we can and will assume that $\Omega\subset \{x\in\mathbb{R}^n:\, x_i>0,\; i=1,\dots,n\}$.

\subsection{Nagel-Stein-Wainger frequency projections}\label{sec:NSW} Given a H\"ormander-Mikhlin multiplier $m$ and $v\in {\mathbb{S}^{n-1}}$ we note that the function $\xi\mapsto m(\xi \cdot v)$ is in general singular on the hyperplane $v^\perp$. It is thus convenient, and very effective, to isolate the singularity of the symbol by the use of suitable cones or wedges.

Let $\omega(\xi)$ denote a function that is homogeneous of degree zero and $C^{\infty}$ away from the origin in $\mathbb{R}^n$, and which satisfies 
\[
{\omega(\xi)}\equiv 
\begin{cases} 1,& \quad\text{if}\quad |\xi_1+\dots+\xi_n|<\frac{1}{2n^2} \|\xi\|,\vspace{.7em}\\
0,&\quad\text{if}\quad|\xi_1+\dots+\xi_n|\ge \frac{1}{n^2}\|\xi\|.
\end{cases}
\]
For a direction {$v\in \mathbb S^{n-1}$} we define the smooth frequency projections
\begin{equation}\label{eq:NSWcones}
 W_v f(x)\coloneqq \int_{\R^n}\omega(v_1\xi_1,\dots,v_n\xi_n)\widehat f(\xi) e^{i x\cdot \xi}\,\d \xi,\qquad x\in\R^n.
\end{equation}
These multipliers were first considered in \cite{NSW}. Note that the operator $\mathrm{Id}-W_v$ is a smooth frequency projection onto a cone with axis along $v$. In particular the frequency support of the symbol of $\mathrm{Id}-W_v$ only intersects the $(n-1)$ dimensional hyperplane $v^\perp$ at the origin.

\tdplotsetmaincoords{60}{120}
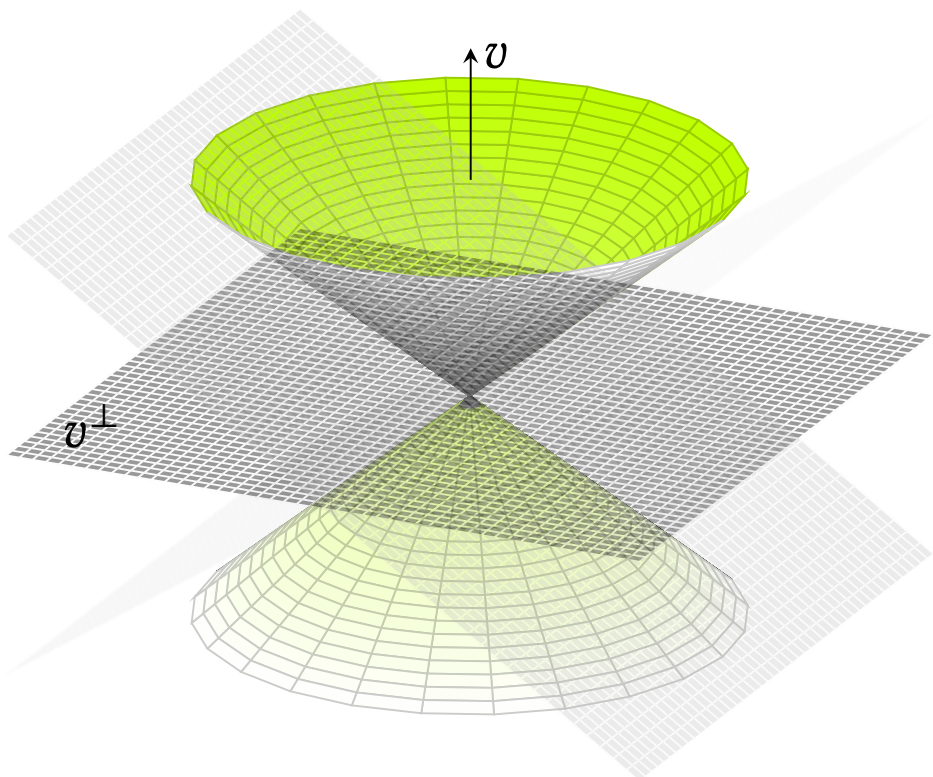
\begin{figure}
\begin{center}
\begin{tikzpicture}
[scale=1.8,
    axis/.style={->,black,thick},
    vector/.style={-stealth,red,very thick},
    vector guide/.style={dashed,red,thick}]
\begin{axis}[
hide axis,
axis lines=center,
domain=0:2,
y domain=0:2*pi,
xmin=-2.5, xmax=2.5,
ymin=-2.5, ymax=2.5,
   mesh/interior colormap=  	{white}{color=(white) color=(lime)},
   colormap/blackwhite,
]
 \addplot3[surf]
  	({x*cos(deg(y))},{x*sin(deg(y))},{x});
    \addplot3[surf]
({-x*cos(deg(y))},{-x*sin(deg(y))},{-x});

\end{axis}

\begin{axis}[
 hide axis,
  axis lines=center,
  zlabel={$e_3$},
   domain=0:2,
   y domain=0:2*pi,
  colormap/blackwhite,
]

\addplot3[surf,fill opacity=0.4, draw opacity=0.01,color=black]
    	({x},{y},{0});
\addplot3[surf,fill opacity=0.4,draw opacity=0.01,color=black]
	      	({-x},{-y},{0});
\addplot3[surf,fill opacity=0.4,draw opacity=0.01,color=black]
	      	({-x},{y},{0});
\addplot3[surf,fill opacity=0.4,draw opacity=0.01,color=black]
		      	({x},{-y},{0});

			    \addplot3[surf,fill opacity=0.3,draw opacity=0.01,color=lightgray]
			      	({x},{-y},{0.5*x});

				\addplot3[surf,fill opacity=0.3,draw opacity=0.01,color=lightgray]
			      	({x},{y},{0.5*x});
				\addplot3[surf,fill opacity=0.3,draw opacity=0.01,color=lightgray]
			      	({-x},{y},{0.5*x});
				\addplot3[surf,fill opacity=0.3,draw opacity=0.01,color=lightgray]
			      	({-x},{-y},{0.5*x});

		    \addplot3[surf,fill opacity=0.3,draw opacity=0.01,color=lightgray]
		      	({x},{-y},{-0.5*x});

			\addplot3[surf,fill opacity=0.3,draw opacity=0.01,color=lightgray]
		      	({x},{y},{-0.5*x});
			\addplot3[surf,fill opacity=0.3,draw opacity=0.01,color=lightgray]
		      	({-x},{y},{-0.5*x});
			\addplot3[surf,fill opacity=0.3,draw opacity=0.01,color=lightgray]
		      	({-x},{-y},{-0.5*x});

			\node at (25,200,1000) {$v^\perp$};
			\coordinate (A) at (255,260,2350);
			\coordinate (B) at (255,260,2950);
			 \draw [->,>=stealth] (A) -- (B);
 			\node at (265,305,2890) {$v$};

]
\end{axis}

\end{tikzpicture}
\end{center}
\caption{The exterior of a Nagel-Stein-Wainger cone with axis $v$, corresponding to the operator $\mathrm{Id}-W_v$, covers the singularity $v^\perp$; a Parcet-Rogers wedge for a single $\sigma$ is also pictured.}
\end{figure}

\subsection{Parcet-Rogers frequency projections} Following \cite{PR2} we define for $\sigma\in\Sigma$ and $\ell\in\Z$ the following two-dimensional wedges
\[
\Psi_{\sigma,\ell}\coloneqq \left\{\xi\in \mathbb{R}^n\backslash e_{\sigma(2)}^{\perp}:\, \frac{2^{-(\ell+1)}}{n}\le\frac{-\xi_{\sigma(1)}}{\xi_{\sigma(2)}}<2^{-\ell}n \right\},
\]
and
\[
\widetilde{\Psi}_{\sigma,\ell}\coloneqq\left\{\xi\in\mathbb{R}^n\backslash e_{\sigma(2)}^{\perp}:\,\frac{2^{-(\ell+1)}}{n+1}\le\frac{-\xi_{\sigma(1)}}{\xi_{\sigma(2)}}<2^{-\ell}(n+1)\right\}.
\]
Take $\kappa$ to be a bump function such that
\[
\kappa\equiv  
\begin{cases} 1 & \textrm{on}\quad [{1}/{2n},n] ,\vspace{.7em}
	\\
	 0 & \textrm{on}\quad [{1} /{2(n+1)},n+1]^c,
\end{cases} 
\] 
and define the Fourier multiplier operators $K_{\sigma,\ell}$ with symbols 
\[
\kappa_{\sigma,\ell}(\xi)\coloneqq \kappa \left(-\frac{\xi_{\sigma(1)}}{2^{-\ell}\xi_{\sigma(2)}}\right),\qquad K_{\sigma,\ell}f\coloneqq (\kappa_{\sigma,\ell}\hat f)^\vee.
\] 
Note that $\kappa_{\sigma,\ell}$ is smooth, identically $1$ on the wedge $\Psi_{\sigma,\ell}$, and identically $0$ off $\,\widetilde{\Psi}_{\sigma,\ell}$. For a subset $\varnothing\neq U\subseteq \Sigma(n)$ we define 
\[
K_{U,\cic{\ell}}\coloneqq \prod_{\sigma \in U}K_{\sigma,\ell_{\sigma}}
\]
with the product symbol being used to denote for compositions of operators in the display above.

The main geometric observation relating the Nagel-Stein-Wainger cones with the Parcet-Rogers wedges is contained in the following lemma, which is an elaboration of a similar statement from \cite{PR2}*{Proof of Theorem A}.

\begin{lemma}[{Inclusion-Exclusion formula}]\label{lem:cover} Let $\{\Omega_{\sigma,\ell}\}$ be a lacunary dissection of {$\Omega\subset {\mathbb{S}^{n-1}}$} and suppose that $v\in \Omega_{\cic{\ell}}$ for some $\cic{\ell}\in \Z^\Sigma$ with $\cic{\ell}=\{\ell_\sigma:\,\sigma\in\Sigma\}$. Then
\[
W_vf = \sum_{\varnothing\neq U\subseteq {\Sigma(n)} }(-1)^{|U|+1} W_vK_{U,\cic{\ell}}f.
\]
\end{lemma}

\begin{proof} Writing $(W_v f)^\wedge\eqqcolon  \omega_v \widehat f$ we note that the support of $\omega_v$ satisfies
	\[
	\mathrm{supp}\, \omega_v \subseteq \Big\{\xi\in\R^n:\, |\xi \cdot v|<\frac1n \max_{1\leq k\leq n}|\xi_k v_k|\Big\}\eqqcolon C_v.
	\]
We read from  \cite{DPPimrn}*{Proof of Lemma 3.2}, together with the assumption that $v\in \Omega_{\cic{\ell}}$,  that
\[
C_v \subseteq   \bigcup_{\sigma\in\Sigma}\Psi_{\sigma,\ell_\sigma}.
\]
The conclusion of the lemma follows from the display above, the inclusion-exclusion formula, and the fact that for each $\sigma\in\Sigma$ and $\ell\in \Z$ the  operator $K_{\sigma,\ell}$ has symbol $\kappa_{\sigma,\ell}$ which is identically 1 on  $\Psi_{\sigma,\ell}$.
\end{proof}

\section{Some auxiliary results} \label{s:w}
We will need some known facts from the weighted theory of maximal directional singular integrals, and in particular, a weighted version of the Chang-Wilson-Wolff principle. The latter allows us to commute a maximum over $N$ multiplier operators with certain  Littlewood-Paley projections, with a controlled loss in $N$. We refer to \cite{DPPimrn}*{\S4}  for a detailed exposition and just recall here the relevant statements.

\subsection{Directional weighted norm inequalities}In order to state these results we briefly introduce directional $A_p$-weights. Given a closed set of directions $\Omega\subset {\mathbb{S}^{n-1}}$ and a non-negative, continuous function $w$ on $\R^n$, we say that $w$ belongs to $A_p ^\Omega$ if $w$ belongs to the one-dimensional class $A_p(\ell_{v})$ for all lines $\ell_{v}$, ${v}\in \Omega$, with uniform bounds. More precisely, if we define
the segments
\[
I(x,t,{v})\coloneqq\{x+s{v}:\, |s|<t\}\subset \R^n,\qquad x\in\R^n, \quad t>0,\quad {v}\in\Omega,
\]
then
\[
[w]_{A_p ^\Omega}\coloneqq \sup_{\substack{x\in\R^n,t>0\\{v}\in\Omega}} \bigg(\,\avgint_{I(x,t,{v})} w\bigg) \bigg(\,\avgint_{I(x,t,{v})} w^{-\frac{1}{p-1}}\bigg),
\]
and $A_p ^\Omega\coloneqq\{w\in C(\R^n):\, [w]_{A_p ^\Omega}<\infty\}.$ Note that we need to consider continuous weights in order to make sense of their restrictions to line segments in $\R^n$. This turns out to be more of a technical nuisance rather than substantial limitation and it is inconsequential for our applications. Finally we write
\[
A_\infty ^\Omega \coloneqq \bigcup_{p>1} A_p ^\Omega.
\]

In the special case that $\Omega=\{e_1,\ldots,e_n\}$ is the standard coordinate basis we just write $A_p ^*$ for the corresponding $A_p$-class.

The following weighted version of the Marcinkiewicz multiplier theorem, due to Kurtz, can be used in several occasions where we need to prove weighted norm inequalities along lacunary sets of directions. We recall the statement of the result for future reference.

\begin{proposition}[Kurtz \cite{kurtz}] Let $m$ be a $C^\infty$ function in $\R^n$ away from the coordinate hyperplanes and assume that $\|m\|_\infty\leq B$. Suppose that for all $0< k \leq n$ we have
	\[
	{\sup_{\xi_{k+1},\ldots,\xi_n}\int_{\rho}\Big|\frac{\partial ^k m(\xi)}{\partial_{\xi_1} \cdots\partial_{\xi_k}}\Big|\,\d\xi_1\cdots\d \xi_k\leq B}
	\]
for all dyadic rectangles $\rho\subset \R^k$, and any permutation of the coordinates $(\xi_1,\ldots,\xi_n)$. Then for all $p\in(1,\infty)$ and all $w\in A_p ^*$ the multiplier operator $T_m(f)\coloneqq (m \widehat f)^\vee$ satisfies the weighted bounds
\[
{\|T_m\|_{L^p(w)}\coloneqq \|T_m:L^p(w)\to L^p(w)\|}\lesssim [w]_{A_p ^*} ^\gamma
\]
where $\gamma=\gamma(p,n,B)$ and the implicit constant is independent of $w$.
\end{proposition}

With this result in hand we can now recall a weighted bound for the wedge multipliers $K_{U,\cic{\ell}}$ associated with a lacunary dissection of the sphere. The proof is a direct application of the theorem of Kurtz above to the operator
\[
f\mapsto \sum_{\cic{\ell}\in\Z^U}\eps_{\cic{\ell}}K_{U,\cic{\ell}}f,
\]
where $\{\eps_{\cic{\ell}}\}$ is an arbitrary choice of signs.
\begin{lemma}\label{lem:weightedPRwedge} Let $\Sigma$ be associated with a given ONB on ${\mathbb{S}^{n-1}}$ and denote by $A_p ^*$ the class of weights corresponding to its coordinate directions. Then for all $w\in A_p ^*$ we have
	\[
	\sup_{U\subseteq \Sigma}\Bigg\| \bigg(\sum_{\cic{\ell}\in\mathbb Z^U}\big|K_{U,\cic{\ell}}f\big |^2\bigg)^\frac12\Bigg\|_{L^p(w)}\lesssim [w]_{A_p ^*} ^\gamma \|f\|_{L^p(w)}
	\]
for some $\gamma=\gamma(p,n)$ and implicit constant independent of $f$ and $w$.
\end{lemma}

In a similar spirit and with an identical proof one can easily provide weighted norm inequalities for the conical multipliers $W_v$ associated with a fixed direction $v\in \R^n$. See also \eqref{eq:derivomega} in \S\ref{sec:proof} below for a similar calculation.

\begin{lemma}\label{lem:l2wNSW}For $v\in {\mathbb{S}^{n-1}}$ let $W_v$ be defined as in \eqref{eq:NSWcones}. Then for all $p\in (1,\infty)$ and all $w\in A_p ^*$ we have
	\[
	\sup_{v\in {\mathbb{S}^{n-1}}}\|W_v\|_{L^p(w)}\lesssim [w]_{A_p ^*} ^\gamma
	\]
for some $\gamma=\gamma(n,p)$ and implicit constant independent of $w$.
\end{lemma}

The previous results imply weighted norm inequalities for the maximal function $\M_\Omega$ along directions of a lacunary set $\Omega\subseteq {\mathbb{S}^{n-1}}$
{\[
\M_\Omega f(x) \coloneqq \sup_{v\in \Omega}\sup_{s>0} \frac{1}{2s}\int_{-s} ^s |f(x+tv)| \, \d t,\qquad x\in\R^n.
\]
}

 The proof of these weighted norm inequalities can be found in \cite{DPPimrn}, however said proof is an adaptation of the corresponding Lebesgue measure argument from \cite{PR2}.

\begin{proposition}Let $\Omega\subset {\mathbb{S}^{n-1}}$ be a set of directions which is lacunary of order $L$, where $L$ is a positive integer, and let $w\in A_p ^\Omega$ be a directional weight with respect to $\Omega$. For all $p\in(1,\infty)$ there exists a constant $\gamma=\gamma(p,n)>0$ such that
	\[
	\|\M_\Omega\|_{L^p(w)}\lesssim [w]_{A_p ^\Omega} ^{\gamma L},
	\]
with implicit constant depending only on $p$, $n$ and the lacunarity order of $\Omega$.
\end{proposition}

The boundedness of the directional maximal function $\M_\Omega$ now allows us to extrapolate weighted norm inequalities from $L^2(w)$ as in \cite{DPPimrn}*{\S4.2}. Namely the following holds.

\begin{proposition}\label{prop:extrap} Let $\Omega\subseteq {\mathbb{S}^{n-1}}$ be a (closed) lacunary set of directions of finite order. Suppose that there exists a $p_0\in(1,\infty)$ and $\gamma>0$ such that for some family of pairs of non-negative function $(f,g)$ we have
	\[
	\|f\|_{L^{p_0}(w)}\lesssim [w]_{A_{p_0} ^\Omega} ^\gamma \|g\|_{L^{p_0}(w)}
	\]
with implicit constant independent of $(f,g)$ and $w$. Then for all $p\in (1,\infty)$ and all $w\in A_p ^\Omega$ we have
\[
\|f\|_{L^p(w)} \lesssim [w]_{A_p ^\Omega} ^{\gamma_p} \|g\|_{L^p(w)}
\]
where $\gamma_p$ depends on $\gamma,n,p$ and the order of lacunarity of $\Omega$; the implicit constant depends only on $p,n$ and the lacunarity order of $\Omega$.
\end{proposition}

\subsection{A maximal inequality for Nagel-Stein-Wainger cones} In the proof of our main theorem we will need a maximal version of Lemma~\ref{lem:l2wNSW}. For this let us consider a set $\Omega\subset {\mathbb{S}^{n-1}}$ and define the maximal cone multiplier operator
\[
W_{\Omega} f(x) \coloneqq \sup_{v\in\Omega}|W_v f(x)|,\qquad x\in \R^n.
\]
\begin{lemma}\label{lem:l2wNSWmax} Let $\Omega\subset {\mathbb{S}^{n-1}}$ be a lacunary set and $w\in A_p ^\Omega$. Then
	\[
	\|W_{\Omega}\|_{L^p(w)} \lesssim [w]_{A_p ^\Omega} ^\gamma 
	\]
for some $\gamma$ depending on $p,n$, and the lacunarity order of $\Omega$.
\end{lemma}

\begin{proof} By the extrapolation result of Proposition~\ref{prop:extrap} it will be enough to proof the $L^2(w)$-version of the conclusion whenever $w\in A_2 ^\Omega$. We will do so by proving the recursive formula
	\[
	\|W_\Omega f\|_{L^2(w)}\leq B [w]_{A_2 ^\Omega} ^\gamma \sup_{\sigma\in \Sigma} \sup_{\ell\in \Z} \|W_{\Omega_{\sigma,\ell}}\|_{L^2(w)}
	\]
with $\gamma$ as in the conclusion of the lemma and $B>0$ a numerical constant depending only upon dimension. The proof then follows by an inductive application of the formula above, repeated as many times as the order of lacunarity $L$ of $\Omega$. The base step of the induction corresponds to lacunary sets of order $0$ in which case the desired estimate is the content of Lemma~\ref{lem:l2wNSW}.

To prove the recursive formula let $v\in \Omega$ so that $v\in \Omega_{\cic{\ell}}$ for some unique $\cic{\ell}\in \Z^\Sigma$. By Lemma~\ref{lem:cover} we have that
{
\[
|W_v f(x)|=\Big| \sum_{\varnothing\neq U\subseteq \Sigma} (-1)^{|U|+1}   W_{\Omega_{\cic{\ell}}} K_{U,\cic{\ell}}f \Big| \lesssim \sup_{\varnothing\neq U\subseteq \Sigma} \sup_{\cic{\ell}\in\Z^\Sigma}|W_{\Omega_{\cic{\ell}}} K_{U,\cic{\ell}}f|
\]
and so 
\[
\|W_{\Omega}f\|_{L^2(w)}\lesssim \big \| \sup_{\varnothing\neq U\subseteq \Sigma} \sup_{\cic{\ell}\in\Z^\Sigma}|W_{\Omega_{\cic{\ell}}} K_{U,\cic{\ell}}f|\big \|_{L^2(w)}\lesssim  \sup_{\varnothing\neq U\subseteq \Sigma} \big \| \sup_{\cic{\ell}\in\Z^\Sigma}|W_{\Omega_{\cic{\ell}}} K_{U,\cic{\ell}}f|\big \|_{L^2(w)}.
\]
The implicit constants in the estimates above depend only on the dimension. Now given $\varnothing\neq U\subseteq \Sigma$ and $\cic{\ell}\in\Z^{\Sigma}$ we write, as in \cite{PR2}*{p. 1545, (6)} $\Z^{\Sigma}=\Z^U\times \Z^{\Sigma\setminus U}$ and $\cic{\ell}\eqqcolon \cic{j}\times \cic{k}$ with $\cic{j}\in\Z^U$ and $\cic{k}\in\Z^{\Sigma\setminus U}$. With this notation in hand we can now estimate for  $\varnothing\neq U\subseteq \Sigma$ and any sequence $\{ f_{\cic{j}}\}_{\cic{j}\in\Z^U}$
\[
\begin{split}
\Big \| \sup_{\cic{j}\times\cic{k}=\cic{\ell}\in\Z^\Sigma}|W_{\Omega_{\cic{\ell}}} f_{\cic{j}}|\Big \|_{L^2(w)}&\leq \bigg(\sum_{\cic{j}\in \Z^U} \Big \| \sup_{\cic{k}\in\Z^{\Sigma\setminus U}}|W_{\Omega_{\cic{\ell}}} f_{\cic {j}}|\Big \|_{L^2(w)} ^2 \bigg)^{\frac12}
\\
&  \leq  \sup_{\cic{j}\in \Z^U}\Big\| \sup_{\cic{k}\in\Z^{\Sigma\setminus U}} W_{\Omega_{\cic{\ell}}}\Big\|_{L^2(w)} \bigg(\sum_{\cic{j}\in\Z^U}\|f_{\cic{j}}\|_{L^2(w)} ^2 \bigg)^{\frac{1}{2}}
\\
& \leq \sup_{\sigma\in\Sigma}\sup_{\ell \in\Z} \|W_{\Omega_{\sigma,\ell}}\|_{L^2(w)} \bigg(\sum_{\cic{j}\in\Z^U}\|f_{\cic{j}}\|_{L^2(w)} ^2 \bigg)^{\frac{1}{2}}.
\end{split}
\]
For $\varnothing \neq U\subset \Sigma$ fixed and $\cic{j}\in\Z^U=\{j_\sigma\}_{\sigma\in U}$ as above we let 
\[
f_{\cic{j}}\coloneqq \prod_{\sigma\in U}K_{\sigma,j_\sigma}f=\prod_{\sigma\in U}K_{\sigma,\ell_\sigma}f=K_{U,\cic{\ell}}f
\]
by the definition of $\cic{j}\in\Z^U$. Using the weighted vector-valued inequality of Lemma~\ref{lem:weightedPRwedge} and the estimates above we get
\[
\|W_{\Omega}f\|_{L^2(w)}\lesssim_n [w]_{A_2 ^*} ^\gamma \sup_{\sigma\in\Sigma}\sup_{\ell \in\Z} \|W_{\Omega_{\sigma,\ell}}\|_{L^2(w)} \|f\|_{L^2(w)}
\]
which is the desired estimate.
}
%
%
%
\end{proof}

\subsection{The Chang-Wilson-Wolff reduction} The proof of our main result relies upon suitable frequency decompositions of the maximal multiplier in hand, with directions in a lacunary set. The main splitting of the operator gives an inner part, including the singular sets of the symbols $m(\xi \cdot v)$ for all $v\in \Omega$, and an outer part which is only singular at the origin. Due to the presence of the supremum in the directions, we cannot however directly use Littlewood-Paley theory to analyze these objects. A familiar tool that has been successfully used in several occasions in the theory of directional singular integrals is a consequence of the Chang-Wilson-Wolff inequality, \cite{cww}. This allows us to commute the supremum over $N$ multipliers with a suitable Littlewood-Paley projection at a $\sqrt{\log N}$-loss.

As we are proving $L^2(w)$-results with the plan to extrapolate to $L^p(w)$, we need a weighted version of the Chang-Wilson-Wolff reduction which we formulate below. For the details of the proof see for example \cite{DPPimrn}*{Proposition 5.2} or \cite{Dem} and the references therein. In order to state this result we introduce a coordinate-wise Littlewood-Paley decomposition in the usual fashion. 

Letting $p$ be a smooth function on $\R$ such that
\[
\sum_{t\in\Z} p(2^{-t}\xi)=1,\qquad \xi \neq 0,
\]
and such that $p$ vanishes off the set $\{\xi\in\R:\, \frac12<|\xi|<2\}$, we define
\[
(P_t ^j f)^\wedge (\xi) \coloneqq p(2^{-t} \xi_j){\hat f(\xi)},\qquad \xi=(\xi_1,\ldots,\xi_n)\in\R^n,\quad \quad t\in\Z.
\]

\begin{proposition}\label{prop:weightedcww} Let  $\{R_1,\ldots,R_N\}$ be Fourier multiplier operators on $\R^n$ satisfying uniform $L^2(w)$-bounds
\[
\sup_{1\leq \tau\leq N}\|R_\tau\|_{L^2(w)} \leq  [w]_{A_2^*}^\gamma
\]
for some $\gamma>0$. Let $\{P_t ^j\}_{t\in\mathbb Z}$ be a smooth Littlewood-Paley decomposition acting on the $j$-th frequency variable, where $1\leq j\leq n$. For $w\in A_p ^*$ and $1<p<\infty$ we then have
	\[
	\Big\|\sup_{1\leq \tau \leq N} |R_\tau f|\Big\|_{L^p(w)}\lesssim  [w]_{A_p^*}^{\gamma_p} 
	\bigg( \|f\|_{L^p(w)} +   \sqrt{\log (N+1)} \, \Big\| \big(\sum_{t\in\mathbb Z}  \sup_{1\leq \tau \leq N} |R_\tau P_t^j f|^2 \big)^\frac{1}{2} \Big\|_{L^p(w)}\bigg)
	\]
for some exponent $\gamma_p=\gamma_p(\gamma,p,n)$ and implicit constant independent of $w,f,N$. 
\end{proposition}

\section{The proof of Theorem~\ref{thm:main}}\label{sec:proof} This section is dedicated to the proof of our main theorem. We remember that $m\in C^\infty(\R^n\backslash\{0\})$ and $T_v$ is the directional multiplier operator
\[
T_vf(x)=\int_{\mathbb{R}^n}\widehat{f}(\xi)m(\xi \cdot v)e^{ix\cdot \xi} \, d\xi,\qquad x\in \R^n,
\]
while for any $\Omega\subset {\mathbb{S}^{n-1}}$ we have defined $T_\Omega f = \sup_{v\in\Omega}|T_v f|$. By the extrapolation result of Proposition~\ref{prop:extrap} the proof of the statement
\[
\sup_{\substack{O \subset \Omega\\\#O=N}}\|T_O f\|_p\lesssim  (\log N)^{1/2}\|f\|_p,\quad p\in (1,\infty),
\]
is reduced to proving that for all $\Omega\subset {\mathbb{S}^{n-1}}$ which are lacunary of some order $L\geq 1$ and all directional weights $w\in A_2 ^\Omega$ we have
\begin{equation}\label{eq:main}
	\sup_{\substack{O \subset \Omega\\\#O=N}}\|T_O f\|_{L^2(w)} \lesssim [w]_{A_2 ^\Omega} ^\gamma (\log N)^{1/2}\|f\|_{L^2(w)}
\end{equation}
for some $\gamma>0$ depending upon dimension  and the order of lacunarity of $\Omega$.
		

\subsection{The main splitting} The whole proof is guided by the following splitting of the operator $T_v$ into two pieces. The first   contains the singularity of $\xi \mapsto m( \xi \cdot v)$, with the   complementary piece given by a Nagel-Stein-Wainger cone as in \S\ref{sec:NSW}
	%
\begin{equation} \label{e:mainsplit}\begin{split}
&  |T_v f(x)|\leq   |T_v W_vf(x)|+  |T_v(\mathrm{Id}-W_v)f(x)|
  \eqqcolon  |T_v ^{\mathrm{in}}f(x)|+ |T_v ^{\mathrm{out}}f(x)|,\qquad x\in\R^n.
\end{split}
\end{equation}
Recall that $W_v$ is defined in \S\ref{sec:NSW}. Surprisingly, the singular inner part is the easiest to deal with, and we treat it first.

\subsubsection*{The inner part} For fixed $v\in O\subset \Omega$ there exists a unique $\cic{\ell}\in \Z^\Sigma$ such that $v\in \Omega_{\cic{\ell}}$. Fixing such $v$ and $\cic{\ell}$ and using Lemma~\ref{lem:cover} we readily see that
{
\[
\begin{split}
|T_v ^{\mathrm{in}}f(x)| & =  \Big|\sum_{\varnothing\neq U\subseteq \Sigma }(-1)^{|U|+1} T_vW_vK_{U,\cic{\ell}}f(x)\Big| \lesssim  \sup_{\varnothing \neq U\subseteq \Sigma} \Big(\sum_{\cic{\ell}}  \sup_{u\in O\cap \Omega_{\cic{\ell}}} |T_{u} W_{u} K_{U,\cic{\ell}}f (x)|^2\Big)^\frac12
\\
& =   \sup_{\varnothing \neq U\subseteq \Sigma}   \Big(\sum_{\cic{\ell}} |T^{\mathrm{in}} _{O\cap \Omega_{\cic{\ell}}} K_{U,\cic{\ell}}f (x)|^2\Big)^\frac12 ,
\end{split}
\]
with implicit constant depending upon dimension,} and where we have implicitly defined the maximal operator
\begin{equation}
\label{e:innerpart}
 T_O ^{\mathrm{in}}f\coloneqq \sup_{v\in O} |T_v ^{\mathrm{in}}f|=\sup_{v\in O}|T_v W_v f|.
\end{equation}
Taking $L^2(w)$-norms and using the weighted vector-valued bound of Lemma~\ref{lem:weightedPRwedge}
\[
\begin{split}
  \left\|T_O ^{\mathrm{in}}f\right\|_{L^2(w)} ^2 &\lesssim [w]_{A_2 ^\Omega} ^{2\gamma_1} \sup_{\cic{\ell}\in\Z^\Sigma}\left\|T^{\mathrm{in}} _{O\cap \Omega_{\cic{\ell}}} \right\|_{L^2(w)} ^2 \|f\|_{L^2(w)} ^2 \\ & \lesssim  [w]_{A_2 ^\Omega} ^{\gamma_2}   \left\|W _{\Omega} \right\|_{L^2(w)} ^2 \sup_{\sigma\in\Sigma}\sup_{\ell \in \Z} \left\|T _{O_{\sigma,\ell}} \right\|_{L^2(w)} ^2 \|f\|_{L^2(w)} ^2.
\end{split}
\]
Inserting the maximal inequality of Lemma~\ref{lem:l2wNSWmax} in the display above proves the recursive estimate
\begin{equation} \label{e:recest}
\|T_O ^{\mathrm{in}}\|_{L^2(w)}   \lesssim  [w]_{A_2 ^\Omega} ^{\tilde \gamma} \sup_{\sigma\in\Sigma}\sup_{\ell \in \Z} \|T _{O_{\sigma,\ell}} \|_{L^2(w)}  
\end{equation}
for some exponent $\tilde \gamma$ depending only on the lacunarity order of $\Omega$ and the dimension.

\subsubsection*{The outer part} Let $\varphi$ to be a bump function on $\mathbb{R}$ such that $\varphi\equiv 0$ on $[-1/4,1/4]$ and $\varphi\equiv 1$ on $(-1/2,1/2)^c$, and define 
\[
\varphi_{v}^j\left(\xi\right)\coloneqq \varphi\left(\frac{nv_j\xi_j}{\|(v\xi)\|}\right),\qquad \xi=(\xi_1,\ldots,\xi_n)\in\R^n\backslash\{0\};
\]
from here on, $(v\xi)$ denotes the vector $(v_1\xi_1,\dots,v_n\xi_n)$. Observe that on $\R^n\backslash\{0\}$ we have
\begin{equation}
\label{e:etas}
\begin{split}
 & 1=  \varphi_{v}^1+ \left( \sum_{j=2}^{n-1}  \varphi_{v}^j \prod_{1\leq \ell <j}  (1-\varphi_{v}^\ell )\right)  + \prod_{1\leq \ell <n}  (1-\varphi_{v}^\ell )
 \eqqcolon  \eta_{v}^1 +\left( \sum_{j=2}^{n-1} \eta_{v}^j\right)+\eta_{v}^n .
\end{split}
\end{equation}
Therefore, we can further split the operator $T_v ^\mathrm{out}=T_v(\mathrm{Id}-W_v)$ into $n$ pieces, 
\begin{equation}\label{eq:decomp}
T_v ^{\mathrm{out}} f=\sum_{j=1}^n T_v ^{\mathrm{out}} N_{v}^j f,
\end{equation}
where each $N_{v}^j$ is the Fourier multiplier with symbol $\eta_{j,v}$. 

The heart of the proof for the outer part is the content of the following lemma which provides a pointwise control of the operators $T_v ^{\mathrm{out}} N_{v}^j P^j _t$ by suitable averages which are independent of the direction. Here $P _t ^j$ is a coordinate-wise Littlewood-Paley projection which is defined as in the discussion preceding Lemma~\ref{prop:weightedcww}. That is, 
\[
(P^j _t f)^\wedge (\xi) \coloneqq p(2^{-t}\xi_j)\hat f(\xi),\qquad \xi=(\xi_1,\ldots,\xi_n)\in\R^n\backslash\{0\},\quad t\in\Z,
\]
with $\mathrm{supp}(p)\subseteq \{\xi \in \R:\, \frac12 <|\xi|<2\}$. We will need to superimpose another Littlewood-Paley decomposition on top of $\{P^j_t\}$. To this aim,  consider a smooth function $q$ on $\R$ such that
\[
\mathrm{supp}(q)\subseteq \{\xi\in\R:\, \frac14<|\xi|<4\},\qquad q\equiv 1\quad \text{on}\quad \{\frac12<|\xi|<2\},
\]
and 
\[
\sum_{t\in\Z}q(2^{-t}\xi){=1},\qquad \xi\in\R^n\backslash\{0\}.
\]
In the statement of the lemma below, $\M_{\mathsf{str}}$ denotes the strong maximal function in $\R^n$, with respect to our fixed choice of coordinates 

\begin{lemma}\label{lem:pointwise} For $v\in {\mathbb{S}^{n-1}}$ and $j=1,\dots,n$, we have the pointwise estimate
\[
|T_v^{\mathrm{out}} N_v^j P^j_t f(x)|\lesssim \M_{\mathsf{str}}( P^j_tf)(x)
\]
with implicit constant depending only upon dimension.
\end{lemma}
\begin{proof} For $v\in\mathbb S^{n-1}$ call 
\[
\Phi_v(x)\coloneqq \int_{\mathbb{R}^n} m(v\cdot\xi)(1-\omega_v(\xi))\eta_v^j(\xi)q(2^{-t}\xi_j ) e^{ix\cdot\xi}\, d\xi,\qquad x\in\R^n.
\]
Remember that $v\in \Omega_{\ell}$ means that for every pair $\sigma=(k,j) $ with $1\le k<j\le n$ we have that $v_j/v_k\sim 2^{-\ell_{(k,j)}}$. Now for a general pair $(k,j)$, call $\ell_{kj}\coloneqq \ell_{(k,j)}$ if $k<j$ and $\ell_{kj}\coloneqq -\ell_{(j,k)}$ if $k>j$. Set also $\ell_{kk}=0$. 
		
From the construction of $\varphi_v^j$, and the definition \eqref{e:etas} of $\eta_v^j$,  it follows that
\[
\xi\in \mathrm{supp}\,\eta_v^j\implies \|(v\xi)\|\lesssim |v_j\xi_j|.
\]
Then, for $k=1,\dots,n$,
\[
|\xi_k|\le\frac{\|(v\xi)\|}{v_k}\lesssim \frac{v_j}{v_k}|\xi_j|\lesssim 2^{t-\ell_{kj}},
\]
which shows that $|\Phi_v(x)|\lesssim \prod_{k=1}^n 2^{t-\ell_{kj}}$. 
			
We proceed to show suitable derivative estimates for the Fourier transform of $\Phi$. Without further mention,   estimates \eqref{eq:derivomega0}, \eqref{eq:derivomega}, and \eqref{eq:derivomega1} are meant to hold for $\xi \in \mathrm{supp} \, \widehat \Phi$, and $\alpha_1,\dots,\alpha_n$ will denote non negative integers with $\alpha=\alpha_1+\dots+\alpha_n$. Firstly,
\begin{equation}\label{eq:derivomega0}
|\partial^{\alpha_1}_{\xi_1}\dots\partial^{\alpha_n}_{\xi_n}\eta_j^v(\xi)|\lesssim\left(\frac{v_1}{v_j}\right)^{\alpha_1}\cdots\left(\frac{v_n}{v_j}\right)^{\alpha_n}\frac{1}{|\xi_j|^{\alpha}} \lesssim \prod_{k=1}^n 2^{{\alpha_k(\ell_{kj}-t)}}.
\end{equation}
It is not difficult to see that $\omega_v$ will satisfy the same derivative estimates, namely 
\begin{equation}\label{eq:derivomega}
|\partial^{\alpha_1}_{\xi_1}\dots\partial^{\alpha_n}_{\xi_n}\omega_v(\xi)|\lesssim \left(\frac{v_1}{\|(v\xi)\|}\right)^{\alpha_1}\dots \left(\frac{v_n}{\|(v\xi)\|}\right)^{\alpha_n}\lesssim \prod_{k=1}^n 2^{\alpha_k({\ell_{kj}-t})}.
\end{equation}
Note that estimate \eqref{eq:derivomega} above was already implicitly used in the proof of Lemma~\ref{lem:l2wNSW}.		
Finally, we have to consider the derivatives of $\xi \mapsto m(\xi \cdot v)$:
\[
|\partial^{\alpha_1}_{\xi_1}\dots\partial^{\alpha_n}_{\xi_n}m(\xi \cdot v)|\le |m^{(\alpha)}(v\cdot\xi)|v_1^{\alpha_1}\dots v_n^{\alpha_n}\lesssim \left(\frac{v_1}{|v\cdot\xi|}\right)^{\alpha_1}\dots\left(\frac{v_n}{|v\cdot\xi|}\right)^{\alpha_n}.
\]
Observe that, since we are taking $\xi\in\mathrm{supp}(1-\omega_v)$, we have that 
\[
|v\cdot\xi|\ge\frac{1}{2n^2}\|(v\xi)\|\gtrsim|v_j\xi_j|
\]
so that as before
\begin{equation}\label{eq:derivomega1}
|\partial^{\alpha_1}_{\xi_1}\dots\partial^{\alpha_n}_{\xi_n}m(v\cdot\xi)|\lesssim \prod_{k=1}^n 2^{\alpha_k({\ell_{kj}-t})}.
\end{equation}
Combining \eqref{eq:derivomega0}, \eqref{eq:derivomega}, and \eqref{eq:derivomega1} {together with a standard integration by parts argument} leads to the bound
\[
|\Phi_v(x)|\lesssim  \prod_{k=1}^n \frac{2^{t-\ell_{kj}}}{(1+2^{t-\ell_{kj}} |x_k|)^2}, 
\]
whence
\[
|T_v ^\mathrm{out}  {N_v ^j }P^j_t f(x)|=|T_v ^\mathrm{out}  {N_v ^j } Q_t ^j P^j_t f(x)|=|\Phi_v*(P^j_t f)(x)|\lesssim \M_{\mathsf{str}}(P^j_t f)(x)
\]
as desired. 
\end{proof}

\subsubsection*{Completing the proof} Recall the main splitting for $T_v$ and the estimate for the inner part. We can  then write, for each $O\subset\Omega$ with $\#O=N$, the estimate
\[
\|T_Of\|_{L^2(w)}\leq B[w]_{A_2 ^\Omega} ^\gamma \sup_{\sigma\in\Sigma}\sup_{\ell\in\Z}\|T_{O_{\sigma,\ell}}\|_{L^2(w)}{\|f\|_{L^2(w)}}+\big\|\sup_{v\in O}|T_v ^\mathrm{out}f|\big\|_{L^2(w)},
\]
where $B $ denotes the implicit constant in the bound \eqref{e:recest}.
{
Using the decomposition \eqref{eq:decomp} and Proposition~\ref{prop:weightedcww}  the second summand can be further estimated as follows
\begin{equation}
\label{e:ao2}
\begin{split}
	&  \Big\|\sup_{v\in O}|T_v ^\mathrm{out}f|\Big\|_{L^2(w)} \lesssim_n\sup_{1\leq j \leq n} \Big\|\sup_{v\in O}| T_v ^\mathrm{out}N_v ^j  f|\Big\|_{L^2(w)} 
	\\
&\quad \lesssim  \sqrt{\log N}[w]_{A_2 ^\Omega} ^{\beta'} \sup_{1\leq j \leq n} \Bigg\|\bigg(\sum_{t\in \Z}   \sup_{v\in O}|P_t ^j (T_v ^\mathrm{out}N_v^j f ) |  ^2 \bigg)^\frac12\Bigg\|_{L^2(w)}
\\
	& \quad  \lesssim \sqrt{\log N}[w]_{A_2 ^\Omega} ^{\beta'} \sup_{1\leq j \leq n} \Bigg\|\bigg(\sum_{t\in \Z}    \M_{\mathsf{str}}(P_t ^j  f)   ^2 \bigg)^\frac12\Bigg\|_{L^2(w)}   \lesssim \sqrt{\log N}[w]_{A_2 ^\Omega} ^{\beta'} \|f\|_{L^2(w)}.
\end{split}
\end{equation}
In passing to the last line of the estimate above we used Lemma~\ref{lem:pointwise} while the last approximate inequality follows by the weighted vector-valued estimates for $\M_{\mathsf{str}}$ and weighted Littlewood-Paley theory.
}
%

Combining the estimates \eqref{e:recest}, \eqref{e:ao2},  we realize that we have proved the following almost orthogonality principle for the maximal directional multiplier $T_O$.

\begin{theorem}\label{thm:almortho}Let $\Omega \subset {\mathbb{S}^{n-1}}$ be a set of directions which contains the coordinate directions. Then for all $w\in A_p ^\Omega$ and every lacunary dissection  $\{S_{\sigma,\ell}\}$ of $\mathbb S^{n-1}$ we have
\[
\sup_{\substack{O\subseteq \Omega\\ \#O\leq N}} \|T_Of\|_{L^2(w)} \leq B [w]_{A_2 ^\Omega} ^\gamma \Big(\sup_{\sigma\in\Sigma}\sup_{\ell\in\Z}\|T_{O_{\sigma,\ell}}\|_{L^2(w)}+\sqrt{\log N}\Big)	\|f\|_{L^2(w)}
\]
for constants $B,\gamma>0$ depending upon dimension  and the order of the lacunary dissection.
\end{theorem}
Our main result Theorem~\ref{thm:main} may be easily derived from Theorem \ref{thm:almortho} by means of the following steps.    First, Theorem \ref{thm:almortho} upgrades to the  $L^2(w)$-estimate
\[\sup_{\substack{O\subseteq \Omega\\ \#O\leq N}} \|T_Of\|_{L^2(w)} \lesssim_L [w]_{A_2 ^\Omega} ^{L\gamma}\sqrt{\log N}	\|f\|_{L^2(w)}
\]
  when $\Omega\subset {\mathbb{S}^{n-1}}$ is a lacunary set of order $L\geq 1$. This is obtained  by induction on the order of lacunarity $L$. Indeed, the case $L=0$ is immediate, as a $0$-th order lacunary set contains exactly one direction. The inductive step follows by using the definition of lacunarity and the almost orthogonality principle of Theorem~\ref{thm:almortho}. Finally the $L^p(w)$-estimate of Theorem~\ref{thm:main} for $p\in(1,\infty)$ is a consequence of  the $L^2(w)$-estimate just proved and the extrapolation result of Proposition~\ref{prop:extrap}.

\section{Concluding remarks and open questions} \label{s:q} In this concluding section we tie back our results to the question of $L^p$-bounds for the Hilbert transform along variable Lipschitz lines by describing a few directions of future investigation.
\subsection{Hilbert transform along lacunary-valued, Lipschitz-truncated fields}

In this context, a natural analogue of Stein's vector field problem described in the introduction is   to ask for sufficient, and possibly necessary conditions on the choice of directions $x\mapsto v(x)$  for the $L^2(w)$ or $L^p$-boundedness of the linearized operator
\[
f\mapsto T_{v(x)} f(x)
\]
\emph{under the assumption that the vector field $v$ takes values in a lacunary set $\Omega$.}  We refer to this question below as \emph{the lacunary vector field problem}. While the latter is undeniably a simpler question then the more renowned unrestricted version, it has the advantage of removing  obstacles related to Besicovitch sets, which, at least in dimension three and higher,  are far from being completely understood.

A closer look at the proof of Theorem \ref{thm:main} shows that   the $L^p$-bound for the inner part \eqref{e:innerpart}, 
as well as  the square function estimate 
\[
 \sup_{1\leq j \leq n} \Bigg\|\bigg(\sum_{t\in \Z}  \sup_{v\in \Omega}|T_v ^\mathrm{out}N_v^j  P_t ^j f  |  ^2 \bigg)^\frac12\Bigg\|_{p} \lesssim    \|f\| _{p}, \qquad 1<p<\infty,
\]
hold with no dependence on the cardinality of $\Omega$, while such dependence must necessarily enter the full operator. One possible sufficient condition in  the lacunary vector field problem is     that $T_{v(\cdot)}$ almost commutes with Littlewood-Paley projections, for instance in the form 
\begin{equation}
\label{e:decouple}
\left\|T_{v(\cdot)} ^\mathrm{out}N_{v(\cdot)}^j f
 \right\|_{p} \lesssim \|f\|_p +  \Bigg\|\bigg(\sum_{t\in \Z}   |T_{v(\cdot)} ^\mathrm{out}N_{v(\cdot)}^j   P_t ^j f  |  ^2 \bigg)^\frac12\Bigg\|_{p},   
 \qquad 1\leq j \leq n,
\end{equation}
 for $ 1<p<\infty.$
This estimate, with $\sqrt{\log N}$ loss, has been obtained via the Chang-Wilson-Wolff inequality in the finite cardinality setting. In dimension two, if we drop the lacunary-valued requirement and instead ask that the vector field $v(\cdot)$ has small Lipschitz constant, and the multiplier entering the definition of $T$ is a truncation of the Hilbert transform at unit  scales,  an almost-commuting estimate of the above type holds for the full operator $T_{v(\cdot)}$; see \cite{DPGTZK}. 

In \cite{GT}, Guo and Thiele have  shown that a sufficient condition for the lacunary vector field estimate to hold when $n=2$ is that $v(x)=  \exp(2\pi i 2^{k(x)})$ where $k(x)=\lfloor\log \lambda(x)\rfloor $ is the truncation of a Lipschitz function $\lambda:\R^2 \to (0,1]$. Note that $v$ takes values in a first order lacunary sequence: a generalization to higher order lacunary-valued    Lipschitz truncated vector fields is given in \cite{DPP}. Both works proceed by establishing, more or less explicitly, analogues of \eqref{e:decouple}, with the simplification that in effect only one Littlewood-Paley decomposition is relevant in dimension two. Our approach to Theorem \ref{thm:main} suggests  that a proof of  \eqref{e:decouple} for suitably defined lacunary-valued    Lipschitz truncated vector fields is feasible, and would lead to sufficient conditions for the lacunary vector field problem in higher dimensions.

\subsection{Extensions to bi-parameter, non-translation invariant kernels} The directional multiplier  $T_v$ of \eqref{e:defTv}  may be thought of as a convolution with a singular kernel which is the tensor product of the one-variable kernel $K=\widehat m$ in direction $v$ with the Dirac delta in the $n-1$ coordinates of $v^\perp$, and may thus be thought of as a bi-parameter, translation invariant Calder\'on-Zygmund kernel. It is then natural to ask whether suitable extensions of Theorem \ref{thm:main} and related results may hold for bi-parameter, and possibly non-translation invariant analogs of \eqref{e:defTv}. A rather  general formulation in this context is the following: let $K$ be a smooth function on  $\R^{1+(n-1)} \times \R^{1+(n-1)}$ minus its diagonal, satisfying standard bi-parameter Calder\'on-Zygmund type assumptions, see for instance \cite[Section 2.1]{HenriBiPar}.  For each $v\in \Omega\subset \mathbb{S}^{n-1}$, let $R_v$ be the rotation mapping $\mathrm{span}\, \{v\}$ to $\R \times \{\vec 0_{\R^{n-1}}\}$ and $v^\perp $ to $\{0\} \times \R^{n-1}$. The interest then lies in the sharp cardinality bounds for the maximal directional singular integral on $\R^n$
\[
T_O f(x) \coloneqq \sup_{v\in O} \left|\mathrm{p.v.} \int_{\R^{1+(n-1)}} f(t,s) K( R_v x, R_v(t,s) ) \, \mathrm{d} t\mathrm{d} s\right|, \qquad x\in \R^n,
\]
when $O$ is a finite subset of a lacunary set $\Omega$.
The translation invariant case, where $K$ is the Fourier transform of a bi-parameter H\"ormander-Mikhlin multiplier, may be more tractable within the tools developed in this article.   Finally, we remark  that  sharp estimates for bi-parameter directional square functions have recently appeared in \cite{DPPMeyer}.
\bibliography{GenLac}
\bibliographystyle{amsplain}
\end{document}

